\documentclass[11pt]{amsart}

\usepackage{amsfonts}
\usepackage{amsmath,amssymb,amsthm}
\usepackage{enumerate,graphicx}
\usepackage[square,numbers]{natbib}
\usepackage{MnSymbol}
\usepackage{mathrsfs,bbm}

\newtheorem{theorem}{Theorem}[section]

\newtheorem{proposition}[theorem]{Proposition}

\theoremstyle{definition}

\newtheorem{definition}[theorem]{Definition}
\newtheorem{example}[theorem]{Example}

\newtheorem{remark}[theorem]{Remark}

\begin{document}

\title[Coloring Fibonacci-Cayley tree]{Coloring Fibonacci-Cayley tree: An application to neural networks}

\keywords{Neural networks, learning problem, Cayley tree, separation property, entropy}
\subjclass{Primary 37A35, 37B10, 92B20}

\author{Jung-Chao Ban}
\address[Jung-Chao Ban]{Department of Applied Mathematics, National Dong Hwa University, Hualien 97401, Taiwan, ROC.}
%\email{jcban@gms.ndhu.edu.tw}

\author[Chih-Hung Chang]{Chih-Hung Chang*}
\thanks{*Author to whom any correspondence should be addressed.} %To whom correspondence should be addressed
\address[Chih-Hung Chang]{Department of Applied Mathematics, National University of Kaohsiung, Kaohsiung 81148, Taiwan, ROC.}
\email{chchang@nuk.edu.tw}

\thanks{This work is partially supported by the Ministry of Science and Technology, ROC (Contract No MOST 105-2115-M-259 -006 -MY2 and 105-2115-M-390 -001 -MY2). The first author is partially supported by National Center for Theoretical Sciences.}

\date{June 19, 2017}
\baselineskip=1.5\baselineskip

% -------------------------------------------------------------
\begin{abstract}
This paper investigates the coloring problem on Fibonacci-Cayley tree, which is a Cayley graph whose vertex set is the Fibonacci sequence. More precisely, we elucidate the complexity of shifts of finite type defined on Fibonacci-Cayley tree via an invariant called entropy. It comes that computing the entropy of a Fibonacci tree-shift of finite type is equivalent to studying a nonlinear recursive system. After proposing an algorithm for the computation of entropy, we apply the result to neural networks defined on Fibonacci-Cayley tree, which reflect those neural systems with neuronal dysfunction. Aside from demonstrating a surprising phenomenon that there are only two possibilities of entropy for neural networks on Fibonacci-Cayley tree, we reveal the formula of the boundary in the parameter space.
\end{abstract}

\maketitle

% -------------------------------------------------------------
\section{Introduction}

For the past few decades, neural networks have been developed to mimic brain behavior so that they are capable of exhibiting complex and various phenomena; they are widely applied in many disciplines such as signal propagation between neurons, deep learning, image processing, pattern recognition, and information technology \cite{AB-1999}.

While the overwhelming majority of neural network models adopt $n$-dimensional lattice as network topology, Gollo \emph{et al.}~\cite{GKC-PCB2009,GKC-PRE2012,GKC-SR2013} propose a neural network with tree structure and show that the dynamic range attains large values. Recent literature also provides shreds of evidence that excitable media with a tree structure performed better than other network topologies \cite{AC-PRE2008,KC-NP2006,LSR-PRL2011}. Mathematically speaking, a tree is a group/semigroup with finitely free generators, and group is one of the most ubiquitous structures found both in mathematics and in nature. A Cayley graph is a graph which represents the structure of a group via the notions of generators of the group. Remarkably, Pomi proposes a neural representation of mathematical group structures which store finite group through their Cayley graphs \cite{Pomi-BMB2016}.

Since neural networks with tree structure come to researcher's attention, it is natural to ask how the complexity of a tree structure neural network can be measured. Alternatively, it is of interest to know how much information a neural network can store. This motivates the investigation of the notion of \emph{tree-shifts} introduced by Aubrun and B\'{e}al \cite{AB-TCS2012, AB-TCS2013}. A tree-shift is a shift space defined on a Cayley tree, and a shift space is a dynamical system which consists of patterns avoiding a set of local patterns (such a set is called a \emph{forbidden set}). Roughly speaking, a tree-shift is a Cayley color tree.

Now that we aim to elaborate the complexity of a tree-shift; that is, the amount of information a tree-shift is capable of carrying. From a mathematical point of view, \emph{entropy} is an invariant that measures the complexity of a system. Lind \cite{Lind-ETDS1984} reveals that the topological entropy of a one-dimensional \emph{shift of finite type} (SFT, a shift space whose forbidden set is finite) is the logarithm of a Perron number, and each Perron number can be realized by a one-dimensional SFT. Moreover, there is a closed formula for the topological entropy of a one-dimensional SFT. Nevertheless, the scenario for multidimensional cases is dramatically different.

Unlike the one-dimensional case, there is no algorithm for the computation of the topological entropy of multidimensional SFTs so far. Hochman and Meyerovitch \cite{HM-AoM2010} demonstrate that the topological entropy of a multidimensional SFT is the infimum of a monotonic recursive sequence of rational numbers. For the estimation of topological entropy, readers are referred to \cite{Friedland-2003, MP-ETDS2013, MP-SJDM2013, MP-1979} and the references therein.

In \cite{AB-TCS2012}, Aubrun and B\'{e}al indicate that tree-shifts constitute an intermediate class in between one-sided and multidimensional shifts. Ban and Chang propose an algorithm for the computation of the entropy of a class of tree-shifts called \emph{tree-shifts of finite type} and reveal that there are only two possibilities for the entropy of cellular neural networks with tree structure \cite{BC-2017,BC-2015a,BCH-2017}. The methodology is extremely novel and infers a possibility of approaching the entropy of multidimensional SFTs. They also demonstrate that the entropy of a tree-shift of finite type is the logarithm of a Perron number and vice versa, which extends Lind's result.

Nowadays, we have known that the entropy of a tree structure neural network reflects its complexity. It is natural to consider the complex nature of a neural network representing disordered brain such as Alzheimer's disease, which is one of the most prevalent neurodegenerative disorders causing dementia and related severe public health concerns. It is well-known that Alzheimer's disease is an irreversible, progressive brain disorder that slowly destroys the patient's memory and thinking skills; although the greatest known risk factor is aging, Alzheimer's disease is not just a disease for elders.

Recent experimental investigations support the hypothesis that the concept of illness progression via neuron-to-neuron transmission and transsynaptic transport of pathogens from affected neurons to anatomically interconnected nerve cells is compatible with the inordinately drawn out prodromal phase and the uniform progression of the pathological process in Alzheimer's disease. Also, a neuron-to-neuron transfer and propagation via seeding offer the most straightforward explanation for both the predictable distribution pattern of the intraneuronal tau lesions in the brain as well as the prolonged rate of disease progression that characterize Alzheimer's disease neuropathologically. Readers who are interested in recent development about Alzheimer's disease are referred to \cite{BT-AN2011,BMK-NRMCB2010,GCT-TN2010} and the references therein for more details.

To investigate the complexity of a system with neuronal dysfunction, we propose neural networks whose topologies are Cayley graphs. The basic idea is that an infected neuron cannot transfer signal correctly. Hence we treat those inherit neurons as dead ones. In other words, if $x$ is a node representing some infected neuron, then no node with prefix $x$ is a node in the graph. To clarify the investigation, we study neural networks whose topology is Fibonacci-Cayley tree, which is a semigroup with generators $\{1, 2\}$ and a relation $2 * 2 = 2$. Namely, the nodes of a Fibonacci-Cayley tree come from the Fibonacci sequence.

In Section 2, we recall some definitions and notions about tree-shifts and extend the concept to define a \emph{Fibonacci tree-shift}, which is a shift space whose topology is Fibonacci-Cayley tree. It follows from the definition that a tree-shift is a dynamical viewpoint of coloring problem on a Cayley tree under a given rule.

Section 3 provides the idea of the entropy of a Fibonacci tree-shift. We demonstrate therein that the computation of the entropy of a Fibonacci tree-shift of finite type is equivalent to studying a nonlinear recursive system and propose an algorithm for computing entropy. It is remarkable that, except for extending the methodology developed in \cite{BC-2017} to Fibonacci tree-shifts of finite type, we can generalize such an algorithm to tree-shifts of finite type with more complex topology. Our result might help in the investigation of the multidimensional coloring problem.

Section 4 applies the results to elucidate the complexity of neural networks on Fibonacci-Cayley tree. The Fibonacci tree-shifts came from neural networks are constrained by the so-called \emph{separation property}; after demonstrating there are only two possibilities of entropy for neural networks on Fibonacci-Cayley tree, we reveal the formula of the boundary in the parameter space.

% -------------------------------------------------------------
\section{Fibonacci-Cayley Tree}

We introduce below some basic notions of symbolic dynamics on Fibonacci-Cayley tree. The main difference between Fibonacci-Cayley tree and infinite Cayley trees is their topologies. We refer to \cite{AB-TCS2012,BC-2015} for an introduction to symbolic dynamics on infinite Cayley trees.

% -------------------------------------------------------------
\subsection{Cayley tree}

Let $\Sigma = \{1, 2, \ldots, d\}$, $d \in \mathbb{N}$, and let $\Sigma^* = \bigcup_{n \geq 0} \Sigma^n$ be the set of words over $\Sigma$, where $\Sigma^0 = \{\epsilon\}$ consists of the empty word $\epsilon$. An \emph{infinite (Cayley) tree} $t$ over a finite alphabet $\mathcal{A}$ is a function from $\Sigma^*$ to $\mathcal{A}$; a \emph{node} of an infinite tree is a word of $\Sigma^*$, and the empty word relates to the \emph{root} of the tree. Suppose $x$ is a node of a tree. Each node $xi$, $i \in \Sigma$, is known as a \emph{child} of $x$ while $x$ is the \emph{parent} of $xi$. A sequence of words $(w_k)_{1 \leq k \leq n}$ is called a \emph{path} if, for all $k \leq n-1$, $w_{k+1} = w_k i_k$ for some $i_k \in \Sigma$ and $w_1 \in \Sigma^*$.

Let $t$ be a tree and let $x$ be a node, we refer $t_x$ to $t(x)$ for simplicity. A \emph{subtree} of a tree $t$ rooted at a node $x$ is the tree $t'$ satisfying $t'_y = t_{xy}$ for all $y \in \Sigma^*$, where $xy = x_1 \cdots x_m y_1 \cdots y_n$ means the concatenation of $x = x_1 \cdots x_m$ and $y_1 \cdots y_n$. Given two words $x = x_1 x_2 \ldots x_i$ and $y = y_1 y_2 \ldots y_j$, we say that $x$ is a \emph{prefix} of $y$ if and only if $i \leq j$ and $x_k = y_k$ for $1 \leq k \leq i$. A subset of words $L \subset \Sigma^*$ is called \emph{prefix-closed} if each prefix of $L$ belongs to $L$. A function $u$ defined on a finite prefix-closed subset $L$ with codomain $\mathcal{A}$ is called a \emph{pattern}, and $L$ is called the \emph{support} of the pattern; a pattern is called an \emph{$n$-block} if its support $L = x \Delta_{n-1}$ for some $x \in \Sigma^*$, where $\Delta_n = \bigcup\limits_{0 \leq i \leq n} \Sigma^i$.

Let $\mathcal{A}^{\Sigma^*}$ be the set of all infinite trees over $\mathcal{A}$. For $i \in \Sigma$, the shift transformations $\sigma_i: \mathcal{A}^{\Sigma^*} \to \mathcal{A}^{\Sigma^*}$ is defined as $(\sigma_i t)_x = t_{ix}$ for all $x \in \Sigma^*$. The set $\mathcal{A}^{\Sigma^*}$ equipped with the shift transformations $\sigma_i$ is called the \emph{full tree-shift} over $\mathcal{A}$. Suppose $w = w_1 \cdots w_n \in \Sigma^*$. Define $\sigma_w = \sigma_{w_n} \circ \sigma_{w_{n-1}} \circ \cdots \circ \sigma_{w_1}$; it follows immediately that $(\sigma_w t)_x = t_{wx}$ for all $x \in \Sigma^*$.

Suppose that $u$ is a pattern and $t$ is a tree. Let $\mathrm{supp} (u)$ denote the support of $u$. We say that $u$ is accepted by $t$ if there exists $x \in \Sigma^*$ such that $u_y = t_{xy}$ for every node $y \in \mathrm{supp} (u)$. In this case, we say that $u$ is a pattern of $t$ rooted at the node $x$. A tree $t$ is said to \emph{avoid} $u$ if $u$ is not accepted by $t$; otherwise, $u$ is called an \emph{allowed pattern} of $t$. Given a collection of finite patterns $\mathcal{F}$, let $\mathsf{X}_{\mathcal{F}}$ denote the set of trees avoiding any element of $\mathcal{F}$. A subset $X \subseteq \mathcal{A}^{\Sigma^*}$ is called a \emph{tree-shift} if $X = \mathsf{X}_{\mathcal{F}}$ for some $\mathcal{F}$, and $\mathcal{F}$ is \emph{a forbidden set} of $X$.

A tree-shift $\mathsf{X}_{\mathcal{F}}$ is called a \emph{tree-shift of finite type} (TSFT) if the forbidden set $\mathcal{F}$ is finite; we say that $\mathsf{X}_{\mathcal{F}}$ is a \emph{Markov tree-shift} if $\mathcal{F}$ consists of two-blocks. Suppose $A_1, A_2, \ldots, A_d$ are binary matrices indexed by $\mathcal{A}$, the \emph{vertex tree-shift} $\mathsf{X}_{A_1, A_2, \ldots, A_d}$ is defined as
\begin{equation}
\mathsf{X}_{A_1, A_2, \ldots, A_d} = \{t \in \mathcal{A}^{\Sigma^*}: A_i (t_x, t_{xi}) = 1 \text{ for all } x \in \Sigma^*, 1 \leq i \leq d\}.
\end{equation}
It follows immediately that each vertex tree-shift is a Markov tree-shift and each Markov tree-shift is a TSFT. In \cite{BC-2015}, the authors indicate that every TSFT is conjugated to a vertex tree-shift.

% -------------------------------------------------------------
\subsection{Fibonacci-Cayley tree}

Let $\Sigma = \{1, 2\}$ and let
$A = \begin{pmatrix}
1 & 1 \\
1 & 0
\end{pmatrix}$.
Denote by $\Sigma_A$ the collection of finite walks of graph representation $G$ (cf.~Figure \ref{fig:Fibonacci-graph}) of $A$ together with the empty word $\epsilon$. A \emph{Fibonacci-Cayley tree} (Fibonacci tree) $t$ over a finite alphabet $\mathcal{A}$ is a function from $\Sigma_A$ to $\mathcal{A}$; in other words, a Fibonacci tree is a pattern whose support is the Fibonacci lattice $\Sigma_A$. A pattern $u$ is called an $n$-block if $\mathrm{supp} (u) = x \Delta_{n-1} \bigcap \Sigma_A$ for some $x \in \Sigma_A$. Notably, there are two different supports for $n$-blocks, say $L_1 = \Delta_{n-1} \bigcap \Sigma_A$ and $L_2 = 1 \Delta_{n-1} \bigcap \Sigma_A$, up to a shift of node.

\begin{figure}
\begin{center}
\includegraphics[scale=0.8,page=1]{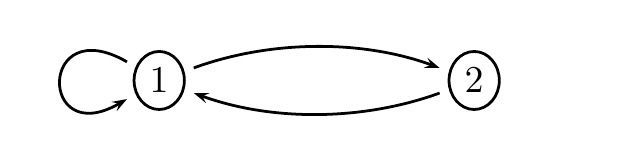}
\end{center}
\caption{Each node of the Fibonacci-Cayley tree is a finite walk in the Fibonacci graph.}
\label{fig:Fibonacci-graph}
\end{figure}

It is seen that the shift transformation $\sigma_2$ is not well-defined in this case. We consider the shift transformation $\sigma: \mathcal{A}^{\Sigma_A} \times \Sigma_A \to \mathcal{A}^{\Sigma_A}$ given by $\sigma (t, x)_y = t_{xy}$ if and only if $xy \in \Sigma_A$.  The set $\mathcal{A}^{\Sigma_A}$ equipped with the new defined shift transformation $\sigma$ is called the \emph{full Fibonacci tree-shift} over $\mathcal{A}$. A subset $X \subseteq \mathcal{A}^{\Sigma_A}$ is called a \emph{Fibonacci tree-shift} if $X = \mathsf{X}_{\mathcal{F}}$ for some forbidden set $\mathcal{F}$.

Similar to the definitions above, a Fibonacci tree-shift $\mathsf{X}_{\mathcal{F}}$ is called a \emph{Fibonacci tree-shift of finite type} (FTSFT) if $\mathcal{F}$ is a finite set. If $\mathcal{F}$ consists of two-blocks, then $\mathsf{X}_{\mathcal{F}}$ is a \emph{Markov-Fibonacci tree-shift}. Suppose $A_1$ and $A_2$ are binary matrices indexed by $\mathcal{A}$, the \emph{vertex-Fibonacci tree-shift} $\mathsf{X}_{A_1, A_2}$ is defined as
\begin{equation}
\mathsf{X}_{A_1, A_2} = \{t \in \mathcal{A}^{\Sigma_A}: A_i(t_x, t_{xi}) = 1 \text{ for all } x, xi \in \Sigma_A\}.
\end{equation}
See Figure \ref{fig:Fibonacci-tree} for the support of a Fibonacci tree.

\begin{figure}[tbp]
\begin{center}
\includegraphics[scale=0.8,page=2]{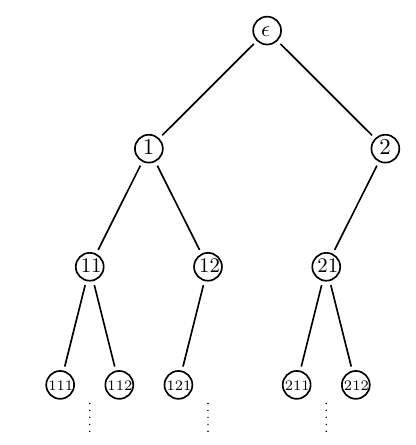}
\end{center}
\caption{The support of Fibonacci-Cayley tree. It is seen that nodes $22, 122, 221, 222, \cdots$ are absent in this case.}
\label{fig:Fibonacci-tree}
\end{figure}

\begin{example}
Suppose that the alphabet $\mathcal{A} = \{R, G\}$ consists of red and green two colors. Let $A_1$ and $A_2$, the coloring rules on the left and right directions, respectively, be given as
$$
A_1 = A_2 = \begin{pmatrix}
1 & 1 \\
1 & 0
\end{pmatrix}.
$$
That is, we can not color green on any two consecutive nodes.
\end{example}

% -------------------------------------------------------------
\section{Complexity of Colored Fibonacci-Cayley Tree}

This section investigates the complexity of Fibonacci tree-shifts of finite type. We use an invariant known as \emph{entropy} introduced in \cite{BC-2017,BC-2015a}.

% -------------------------------------------------------------
\subsection{Entropy of Fibonacci tree-shifts}

Let $X$ be a Fibonacci tree-shift. For $n \in \mathbb{N}$ and $w \in \Sigma_A$, let $\Gamma^{[w]}_n(X)$ denote the set of $n$-blocks of $X$ rooted at $w$. More explicitly,
$$
\Gamma^{[w]}_n(X) = \{u: u \text{ is allowed and } \mathrm{supp}(u) = w \Delta_{n-1} \bigcap \Sigma_A\}.
$$
Fix $c \in \mathcal{A}$, set $\Gamma^{[w]}_{c; n}(X) = \{u \in \Gamma^{[w]}_n(X): u_w = c\}$ and $\gamma^{[w]}_{c; n} = |\Gamma^{[w]}_{c; n}(X)|$, where $| \cdot|$ means the cardinality. For simplicity, we refer to $\Gamma^{[\epsilon]}_{c; n}(X)$ and $\gamma^{[\epsilon]}_{c; n}$ as $\Gamma_{c; n}(X)$ and $\gamma_{c; n}$, respectively. The \emph{entropy} of $X$ is defined as follows.

\begin{definition}
Suppose $X$ is a Fibonacci tree-shift. The \emph{entropy} of $X$, denoted by $h(X)$, is defined as 
\begin{equation}\label{eq:entropy-Fibonacci-tree-shift}
h(X)=\lim_{n\rightarrow \infty} \frac{\ln^2 \gamma_n}{n}
\end{equation}
provided the limit exists, where $\ln^2 = \ln \circ \ln $.
\end{definition}

Notably, the growth rate of the number of nodes of $\Delta_n \bigcap \Sigma_A$ is exponential; this makes the speed of increase of $\gamma_n$ doubly exponential. Hence, the entropy $h(X)$ measures the growth rate of feasible patterns concerning their height.

Let $\mathcal{A} = \{c_1, c_2, \ldots, c_k\}$ for some $k \geq 2 \in \mathbb{N}$. Given two binary matrices $A_1 = (a_{i, j})_{i, j = 1}^k$ and $A_2 = (b_{i, j})_{i, j = 1}^k$, it comes immediately that $\gamma_n = \sum\limits_{\ell=1}^k \gamma_{\ell; n}$ and
\begin{equation}\label{eq:gamma-in-original-formula}
\gamma_{i; n} = \sum\limits_{j_1, j_2 = 1}^k a_{i, j_1} b_{i, j_2} \gamma_{j_1; n-1}^{[1]} \gamma_{j_2; n-1}^{[2]}
\end{equation}
for $1 \leq i \leq k$ and $n \geq 3 \in \mathbb{N}$, herein $X = \mathsf{X}_{A_1, A_2}$ is a vertex-Fibonacci tree-shift, $\gamma_{\ell; n}^{[w]}$ refers to $\gamma_{c_{\ell}; n}^{[w]}$, and $\gamma_{i; 2} = (\sum\limits_{j=1}^k a_{i, j}) \cdot (\sum\limits_{j=1}^k b_{i, j})$ for $1 \leq i \leq k$.

For $w \in \Sigma_A$, the \emph{follower set} of $w$ is defined as
$$
F_w = \{w' \in \Sigma_A: w w' \in \Sigma_A\}.
$$
Since each element $w = w_1 w_2 \cdots w_n \in \Sigma_A$ is a finite walk in the Fibonacci graph $G$ (cf.~Figure \ref{fig:Fibonacci-graph}), it is easily seen that
$$
F_w = \left\{\begin{aligned}
&\Sigma_A, && w_n = 1; \\
&F_2, && \hbox{otherwise.}
\end{aligned}\right.
$$
This makes
\begin{equation}\label{eq:gamma-reduce-formula}
\gamma_{i; n}^{[1]} = \gamma_{i; n} \quad \text{and} \quad \gamma_{i; n}^{[2]} = \sum_{j=1}^k a_{i, j} \gamma_{j; n-1}^{[21]} = \sum_{j=1}^k a_{i, j} \gamma_{j; n-1}
\end{equation}
for $1 \leq i \leq k$ and $n \geq 3 \in \mathbb{N}$. Substituting \eqref{eq:gamma-in-original-formula} with \eqref{eq:gamma-reduce-formula} derives
\begin{equation*}
\gamma_{i; n} = \sum\limits_{j_1, j_2, j_3 = 1}^k a_{i, j_1} b_{i, j_2} a_{j_2, j_3} \gamma_{j_1; n-1} \gamma_{j_3; n-2}
\end{equation*}
for $1 \leq i \leq k$ and $n \geq 4 \in \mathbb{N}$. Alternatively, computing the entropy of $X$ is equivalent to the investigation of the following nonlinear recursive system.
\begin{equation}\label{eq:gamma-recursive-formula}
\left\{\begin{aligned}
\gamma_{i; n} &= \sum\limits_{j_1, j_2, j_3 = 1}^k a_{i, j_1} b_{i, j_2} a_{j_2, j_3} \gamma_{j_1; n-1} \gamma_{j_3; n-2}, \\
\gamma_{i; 2} &= (\sum\limits_{j=1}^k a_{i, j}) \cdot (\sum\limits_{j=1}^k b_{i, j}),
\end{aligned}\right.
\end{equation}
for $1 \leq i \leq k$ and $n \geq 4 \in \mathbb{N}$. We call \eqref{eq:gamma-recursive-formula} the \emph{recurrence representation} of $\mathsf{X}_{A_1, A_2}$.

\begin{example}\label{eg:A12-golden-mean}
Suppose that $\mathcal{A} = \{c_1, c_2\}$ and
$$
A_1 = A_2 = \begin{pmatrix}
1 & 1 \\
1 & 0
\end{pmatrix}.
$$
Observe that
$$
\left\{
\begin{aligned}
\gamma_{1; n} &= (\gamma_{1; n-1}^{[1]} + \gamma_{2; n-1}^{[1]}) \cdot (\gamma_{1; n-1}^{[2]} + \gamma_{2; n-1}^{[2]}), \\
\gamma_{2; n} &= \gamma_{1; n-1}^{[1]} \cdot \gamma_{1; n-1}^{[2]}, \\
\gamma_{1; 2} &= 4, \gamma_{2; 2} = 1.
\end{aligned}
\right.
$$
Since $\gamma_{i; n}^{[1]} = \gamma_{i; n}$,
$$
\gamma_{1; n}^{[2]} = \gamma_{1; n-1}^{[21]} + \gamma_{2; n-1}^{[21]} = \gamma_{1; n-1} + \gamma_{2; n-1},
$$
and
$$
\gamma_{2; n}^{[2]} = \gamma_{1; n-1}^{[21]} = \gamma_{1; n-1},
$$
examining the entropy of $\mathsf{X}_{A_1, A_2}$ is equivalent to studying
$$
\left\{
\begin{aligned}
\gamma_{1; n} &= (\gamma_{1; n-1} + \gamma_{2; n-1}) \cdot (2\gamma_{1; n-2} + \gamma_{2; n-2}), \\
\gamma_{2; n} &= \gamma_{1; n-1} \cdot (\gamma_{1; n-2} + \gamma_{2; n-2}), \\
\gamma_{1; 2} &= 4, \gamma_{2; 2} = 1.
\end{aligned}
\right.
$$
\end{example}

In the next subsection, we introduce an algorithm for solving the growth rate of the nonlinear recursive system \eqref{eq:gamma-recursive-formula}.

% -------------------------------------------------------------
\subsection{Computation of entropy}

The previous subsection reveals that computing the entropy of a Fibonacci tree-shift of finite type is equivalent to studying a corresponding nonlinear recursive system. To address an algorithm for the computation of the entropy, we start with the following proposition.

\begin{proposition}\label{prop:entropy-exist-TSFT}
Suppose that $X = \mathsf{X}_{A_1, A_2}$ is a vertex-Fibonacci tree-shift over $\mathcal{A}$ with respect to $A_1, A_2$. Then the limit \eqref{eq:entropy-Fibonacci-tree-shift} exists and
\begin{equation}\label{eq:lnln-sum-gamma-equal-ln-sum-ln-gamma}
h(X) = \lim_{n \to \infty} \dfrac{\ln \sum_{i=1}^k \ln \gamma_{i; n}}{n}.
\end{equation}
\end{proposition}
\begin{proof}
The existence of limit \eqref{eq:entropy-Fibonacci-tree-shift} comes immediately from the subadditivity of $\ln^2 \gamma_n$. Readers are referred to \cite[Chapter 4]{LM-1995} for related discussion. From the Cauchy inequality
$$
\gamma_{1; n} \cdot \gamma_{2; n} \cdots \gamma_{k; n} \leq \left( \dfrac{\sum_{i=1}^k \gamma_{i; n}}{k} \right)^k,
$$
we derive that
$$
\sum_{i=1}^k \ln \gamma_{i; n} \leq k (\ln \sum_{i=1}^k \gamma_{i; n} - \ln k).
$$
It can be verified without difficulty that
$$
\lim_{n \to \infty} \dfrac{\ln \sum_{i=1}^k \ln \gamma_{i; n}}{n} \leq \lim_{n \to \infty} \dfrac{\ln^2 \sum_{i=1}^k \gamma_{i; n}}{n} = h(X).
$$

It is noteworthy that if there exist $1 \leq i \leq k$ and $m \geq 2$ such that $\gamma_{i; m} \geq \gamma_{j; m}$ for all $j \neq i$, then $\gamma_{i; n} \geq \gamma_{j; n}$ for all $j \neq i$ and $n \geq m \in \mathbb{N}$. Without loss of generality, we may assume that $\gamma_{1; n} \geq \gamma_{i; n}$ for all $n \geq 2 \in \mathbb{N}$ and $1 \leq i \leq k$. The inequality
$$
\gamma_{1; n} \leq \sum_{i=1}^k \gamma_{i; n} \leq k \gamma_{1; n}
$$
reveals that
$$
\lim_{n \to \infty} \dfrac{\ln^2 \sum_{i=1}^k \gamma_{i; n}}{n} = \lim_{n \to \infty} \dfrac{\ln^2 \gamma_{1; n}}{n}.
$$
Meanwhile,
$$
\sum_{i=1}^k \ln \gamma_{i; n} = \ln \prod_{i=1}^k \gamma_{i; n} \geq \ln \gamma_{1; n}.
$$
Therefore, we conclude that
$$
\lim_{n \to \infty} \dfrac{\ln \sum_{i=1}^k \ln \gamma_{i; n}}{n} \geq \lim_{n \to \infty} \dfrac{\ln^2 \gamma_{1; n}}{n} = \lim_{n \to \infty} \dfrac{\ln^2 \sum_{i=1}^k \gamma_{i; n}}{n} = h(X).
$$
This completes the proof.
\end{proof}

A symbol $c \in \mathcal{A}$ is called \emph{essential} (resp.~\emph{inessential}) if $\gamma_{c; n} \geq 2$ for some $n \in \mathbb{N}$ (resp.~$\gamma_{c; n} = 1$ for all $n \in \mathbb{N}$). Proposition \ref{prop:entropy-exist-TSFT} infers the alphabet $\mathcal{A}$ can be expressed as the disjoint union of two subsets $\mathcal{A}_E$ and $\mathcal{A}_I$ consisting of essential and inessential symbols, respectively. In addition, it is easily seen from \eqref{eq:lnln-sum-gamma-equal-ln-sum-ln-gamma} that
$$
h(X) = \lim_{n \to \infty} \frac{\ln^2 \sum_{c \in \mathcal{A}} \gamma_{c; n}}{n} = \lim_{n \to \infty} \frac{\ln^2 \sum_{c \in \mathcal{A}_E} \gamma_{c; n}}{n}.
$$
For the simplicity, we assume that $\mathcal{A} = \mathcal{A}_E$; that is, each symbol is essential.

Now we turn back to the elaboration of the nonlinear recursive system \eqref{eq:gamma-recursive-formula}. First, we consider a subsystem
\begin{equation}\label{eq:subsystem-gamma-recursive}
\gamma_{i; n} = \gamma_{i_{j_1}; n-1} \gamma_{i_{j_2}; n-1}, \quad 1 \leq i \leq k,
\end{equation}
arose from \eqref{eq:gamma-recursive-formula}. Notably, for each $i$, the above subsystem contains only one term in the original system and the coefficient is $1$. We call such a system \emph{simple}. Generally speaking, a Fibonacci tree-shift has many simple recurrence representations. The following theorem reveals the entropy of a vertex-Fibonacci tree-shift with simple recurrence representation relates to the spectral radius of some integral matrix.

\begin{theorem}\label{thm:entropy-simple-system}
Suppose that $X = \mathsf{X}_{A_1, A_2}$ is a Fibonacci tree-shift over $\mathcal{A}$ such that each symbol is essential. If the recurrence representation of $X$ is simple, then there exists an integral matrix $M$ such that 
\begin{equation*}
h(X) = \log \rho_M,
\end{equation*}
where $\rho_M$ is the spectral radius of $M$.
\end{theorem}
\begin{proof}
Since $c_{i}$ is essential for $1\leq i\leq k$, we may assume that $\gamma_{i;2} \geq 2$ for simplicity. Proposition \ref{prop:entropy-exist-TSFT} demonstrates that 
\begin{equation*}
h(X) = \lim_{n \to \infty} \frac{\ln \sum_{i=1}^{k} \ln \gamma_{i;n}}{n}.
\end{equation*}
Let
\begin{equation}\label{eq:def-theta-n}
\theta_{n} = ( \ln \gamma_{1;n}, \ln \gamma_{1;n-1}, \ln \gamma_{2;n}, \ln \gamma_{2;n-1}, \ldots, \ln \gamma_{k;n}, \ln \gamma_{k;n-1})^{T}
\end{equation}
be a $2k \times 1$ vector. The recursive system \eqref{eq:gamma-reduce-formula} suggests that there exists a $2k\times 2k$ nonnegative integral matrix $M$ such that 
\begin{equation}
\theta_{n} = M \theta_{n-1} \quad \text{for} \quad n \geq 3.
\end{equation}
Note that $\gamma_{i; 2} = (\sum\limits_{j=1}^k a_{i, j}) \cdot (\sum\limits_{j=1}^k b_{i, j}) \leq k^2$ for $1 \leq i \leq k$. Therefore,
\begin{align*}
h(X) &= \lim_{n \to \infty} \frac{\ln \sum_{i=1}^{k} \ln \gamma_{i;n}}{n} \\
 &\leq \lim_{n \to \infty} \frac{\ln \sum_{i,j=1}^{2k} M^{n}(i,j)}{n} \leq \ln \rho_{M}.
\end{align*}
On the other hand,
\begin{align*}
h(X) &= \lim_{n \to \infty} \frac{\ln \sum_{i=1}^{k} \ln \gamma_{i;n}}{n} \\
 &= \lim_{n \to \infty} \frac{\ln \sum_{i=1}^{k} 2 \ln \gamma_{i;n}}{n} \\
 &\geq \lim_{n \to \infty}\frac{\ln \sum_{i=1}^{k} (\ln \gamma_{i;n} + \ln \gamma_{i;n-1})}{n} \\
 &= \lim_{n \to \infty} \frac{\ln \sum_{i,j=1}^{2k} M^{n}(i,j)}{n} = \ln \rho_{M}.
\end{align*}%
This completes the proof.
\end{proof}

We call the matrix $M$ sketched in Theorem \ref{thm:entropy-simple-system} the \emph{adjacency matrix} of the simple recursive system \eqref{eq:gamma-reduce-formula}. This makes Theorem \ref{thm:entropy-simple-system} an extension of a classical result in symbolic dynamical systems.

\begin{example}\label{eg:A12-golden-mean-p2}
Suppose that the alphabet $\mathcal{A}$ and $A_1, A_2$ are given, and the recurrence representation of $\mathsf{X}_{A_1, A_2}$ is
\begin{equation} \label{eq:eg-simple-system}
\left\{
\begin{aligned}
\gamma_{1; n} &= \gamma_{1; n-1} \cdot \gamma_{1; n-2}, \\
\gamma_{2; n} &= \gamma_{1; n-1} \cdot \gamma_{2; n-2}. 
\end{aligned}
\right.
\end{equation}
Furthermore, each symbol in $\mathcal{A}$ is essential. Let
$
\theta_n = \begin{pmatrix}
\ln \gamma_{1;n} \\
\ln \gamma_{1;n-1} \\
\ln \gamma_{2;n} \\
\ln \gamma_{2;n-1}
\end{pmatrix}.
$
It comes immediately that
$$
M = \begin{pmatrix}
1 & 1 & 0 & 0 \\
1 & 0 & 0 & 0 \\
1 & 0 & 0 & 1 \\
0 & 0 & 1 & 0
\end{pmatrix}
$$
since $\theta_n = M \theta_{n-1}$. In addition, the characteristic polynomial of $M$ is $f(\lambda) = (\lambda - 1) (\lambda + 1)(\lambda^2 - \lambda - 1)$ and the spectral radius of $M$ is $\rho_M = g$, where $g = \dfrac{1 + \sqrt{5}}{2}$ is the golden mean. Theorem \ref{thm:entropy-simple-system} asserts that the entropy of $\mathsf{X}_{A_1, A_2}$ is $h(\mathsf{X}_{A_1, A_2}) = \ln g$.
\end{example}

\begin{theorem}\label{thm:entropy-Fibonacci-tree-shift}
Suppose that $X = \mathsf{X}_{A_1, A_2}$ is a Fibonacci tree-shift over $\mathcal{A}$ such that each symbol is essential. Then
\begin{equation}
h(X) = \max \{\ln \rho_M: M \text{ is the adjacency matrix of a simple subsystem of } X\}.
\end{equation}
\end{theorem}

The main idea of the proof of Theorem \ref{thm:entropy-Fibonacci-tree-shift} can be examined via the following example.

\begin{example}\label{eg:A12-golden-mean-p3}
Suppose that $\mathcal{A} = \{c_1, c_2\}$ and
$$
A_1 = A_2 = \begin{pmatrix}
1 & 1 \\
1 & 0
\end{pmatrix}
$$
are the same as considered in Example \ref{eg:A12-golden-mean}. Recall that the recurrence representation of $\mathsf{X}_{A_1, A_2}$ is
$$
\left\{
\begin{aligned}
\gamma_{1; n} &= (\gamma_{1; n-1} + \gamma_{2; n-1}) \cdot (2\gamma_{1; n-2} + \gamma_{2; n-2}), \\
\gamma_{2; n} &= \gamma_{1; n-1} \cdot (\gamma_{1; n-2} + \gamma_{2; n-2}), \\
\gamma_{1; 3} &= 15, \gamma_{2; 3} = 8, \gamma_{1; 2} = 4, \gamma_{2; 2} = 1.
\end{aligned}
\right.
$$
It can be verified without difficulty that
$$
\gamma_{1; n} \geq \gamma_{2; n} \geq \gamma_{1; n-1} \geq \gamma_{2; n-1} \quad \text{for} \quad n \geq 3.
$$
Therefore, we have two inequalities
\begin{align*}
\gamma_{1; n} &= (\gamma_{1; n-1} + \gamma_{2; n-1}) \cdot (2\gamma_{1; n-2} + \gamma_{2; n-2}) \\
 &= 2 \gamma_{1; n-1} \gamma_{1; n-2} + \gamma_{1; n-1} \gamma_{2; n-2} + 2 \gamma_{2; n-1} \gamma_{1; n-2} + \gamma_{2; n-1} \gamma_{2; n-2} \\
 &= \gamma_{1; n-1} \gamma_{1; n-2} \left( 2 + \frac{\gamma_{2; n-2}}{\gamma_{1; n-2}} + 2 \frac{\gamma_{2; n-1}}{\gamma_{1; n-1}} + \frac{\gamma_{2; n-1} \gamma_{2; n-2}}{\gamma_{1; n-1} \gamma_{1; n-2}} \right) \leq 6 \gamma_{1; n-1} \gamma_{1; n-2},
\end{align*}
and
\begin{align*}
\gamma_{2; n} &= \gamma_{1; n-1} \cdot (\gamma_{1; n-2} + \gamma_{2; n-2}) \\
 &= \gamma_{1; n-1} \gamma_{1; n-2} + \gamma_{1; n-1} \gamma_{2; n-2} \\
 &= \gamma_{1; n-1} \gamma_{1; n-2} \left( 1 + \frac{\gamma_{2; n-2}}{\gamma_{1; n-2}} \right) \leq 2 \gamma_{1; n-1} \gamma_{1; n-2} < 6 \gamma_{1; n-1} \gamma_{1; n-2}.
\end{align*}
Consider the nonlinear recursive system
$$
\alpha_{i; n} = 6 \alpha_{1; n-1} \alpha_{1; n-2}, \quad i = 1, 2, n \geq 3,
$$
let
$$
\theta_n = ( \ln \alpha_{1;n}, \ln \alpha_{1;n-1}, \ln \alpha_{2;n}, \ln \alpha_{2;n-1})^T
$$
and
$$
M = \begin{pmatrix}
1 & 1 & 0 & 0 \\
1 & 0 & 0 & 0 \\
1 & 1 & 0 & 0 \\
0 & 0 & 1 & 0
\end{pmatrix}.
$$
It is noteworthy that $M$ is also the adjacency matrix of the simple recurrence representation
$$
\gamma_{i; n} = \gamma_{1; n-1} \gamma_{1; n-2}, \quad i = 1, 2, n \geq 3,
$$
of $X$. Meanwhile,
$$
\theta_n = M \theta_{n-1} + (\ln 6) \mathbf{1}_4 = M^{n-2} \theta_2 + (\ln 6) (M^{n-3} + \cdots + M + I_4) \mathbf{1}_4,
$$
where $\mathbf{1}_4 = (1, 1, 1, 1)^T$. Therefore, we derive from Proposition \ref{prop:entropy-exist-TSFT} that
\begin{align*}
h(\mathsf{X}_{A_1, A_2}) &= \lim_{n \to \infty} \frac{\ln \sum_{i=1}^2 \ln \gamma_{i;n}}{n} \\
 &\leq \lim_{n \to \infty} \frac{\ln \sum_{i=1}^2 \ln \alpha_{i;n}}{n} \leq \lim_{n \to \infty} \frac{\ln \sum_{i=1}^{n-2} \left\Vert M^{i} \right\Vert}{n} = \ln g,
\end{align*}
where $g = \dfrac{1 + \sqrt{5}}{2}$ is the golden mean.

Obviously, $h(\mathsf{X}_{A_1, A_2}) \geq \ln g$. This concludes that $h(\mathsf{X}_{A_1, A_2}) = \ln g$.
\end{example}

\begin{proof}[Proof of Theorem \ref{thm:entropy-Fibonacci-tree-shift}]
Since every symbol is essential, we may assume that, for $1 \leq i \leq k$, $\gamma_{i;2} \geq 2$ for simplicity. Set
$$
\hbar = \max \{\ln \rho_M: M \text{ is the adjacency matrix of a simple subsystem of } X\}.
$$
It suffices to show that $h(X) \leq \hbar$. Since $h(X)$ exists (cf.~Proposition \ref{prop:entropy-exist-TSFT}), we may assume without losing the generality that 
\begin{equation}
\gamma_{1;n}\geq \gamma_{2;n}\geq \cdots \geq \gamma _{k;n}\geq\gamma _{1;n-1}\geq  \gamma _{2;n-1}\geq \cdots \geq \gamma _{k;n-1}
\end{equation}
for $n \geq 3$. Recall that, for $1 \leq i \leq k$ and $n \geq 4$,
$$
\gamma_{i; n} = \sum\limits_{j_1, j_2, j_3 = 1}^k a_{i, j_1} b_{i, j_2} a_{j_2, j_3} \gamma_{j_1; n-1} \gamma_{j_3; n-2},
$$
where $A_1 = (a_{i, j})_{1 \leq i, j \leq k}, A_2 = (b_{i, j})_{1 \leq i, j \leq k}$. Let $\gamma_{\bar{i}_1; n-1} \gamma_{\bar{i}_2; n-2}$ be the maximum in the above equation. Then
\begin{align*}
\gamma_{i; n} &= \sum\limits_{j_1, j_2, j_3 = 1}^k a_{i, j_1} b_{i, j_2} a_{j_2, j_3} \gamma_{j_1; n-1} \gamma_{j_3; n-2} \\
 &= \gamma_{\bar{i}_1; n-1} \gamma_{\bar{i}_2; n-2} \sum\limits_{j_1, j_2, j_3 = 1}^k a_{i, j_1} b_{i, j_2} a_{j_2, j_3} \frac{\gamma_{j_1; n-1} \gamma_{j_3; n-2}}{\gamma_{\bar{i}_1; n-1} \gamma_{\bar{i}_2; n-2}} \\
 &\leq \kappa \gamma_{\bar{i}_1; n-1} \gamma_{\bar{i}_2; n-2}
\end{align*}
for some constant $\kappa \in \mathbb{N}$. Notably, $\kappa$ is independent of $i$ and $n$. Consider the nonlinear recursive system
\begin{equation}\label{eq:proof-main-thm-subsystem-with-kappa}
\alpha_{i; n} = \kappa \alpha_{\bar{i}_1; n-1} \alpha_{\bar{i}_2; n-2} \quad \text{for} \quad 1 \leq i \leq k.
\end{equation}
Let
$$
\theta_n = ( \ln \alpha_{1;n}, \ln \alpha_{1;n-1}, \ln \alpha_{2;n}, \ln \alpha_{2;n-1}, \ldots, \ln \alpha_{k;n}, \ln \alpha_{k;n-1})^T.
$$
Similar to the discussion in the proof of Theorem \ref{thm:entropy-simple-system}, there exists a $2k\times 2k$ integral matrix $M$ such that
\begin{align*}
\theta_n &= M \theta_{n-1} + (\ln \kappa) \mathbf{1}_{2k} \\
 &= M^{n-2} \theta_2 + (\ln \kappa) (M^{n-3} + \cdots + M + I_{2k}) \mathbf{1}_{2k},
\end{align*}
where $\mathbf{1}_{2k} = (1, \ldots, 1)^T$ is a $2k \times 1$ vector. Combining the computation above with the fact $\gamma_{i;n}\geq 2$ indicates that
\begin{align*}
h(X) &= \lim_{n \to \infty} \frac{\ln^2 \sum_{i=1}^k \gamma_{i;n}}{n} = \lim_{n \to \infty} \frac{\ln \sum_{i=1}^k \ln \gamma_{i;n}}{n} \\
 &\leq \lim_{n \to \infty} \frac{\ln \sum_{i=1}^k \ln \alpha_{i;n}}{n} \leq \lim_{n \to \infty} \frac{\ln \sum_{i=1}^{n-2} \left\Vert M^{i} \right\Vert}{n} = \ln \rho_M.
\end{align*}
Observe that \eqref{eq:proof-main-thm-subsystem-with-kappa} is a simple recurrence representation of $X$ if we replace $\kappa$ in \eqref{eq:proof-main-thm-subsystem-with-kappa} by $1$, and $M$ is the adjacency matrix of such a new system. Therefore, we conclude that $h(X) \leq \hbar$.

The proof is thus complete.
\end{proof}

\begin{remark}
A key presumption in Theorem \ref{thm:entropy-Fibonacci-tree-shift} is that every symbol is essential. If there are inessential symbols, the proof of Theorem \ref{thm:entropy-Fibonacci-tree-shift} infers that we only need to replace $M$ by $M'$, where $M'$ is obtained by deleting all the rows and columns indexed by those inessential symbols.
\end{remark}

% -------------------------------------------------------------
\section{Neural Networks on Fibonacci-Cayley Tree}

In this section, we apply the algorithm developed in the previous section to neural networks with the Fibonacci-Cayley tree as their underlying space. For some reasoning, the overwhelming majority of considered underlying spaces of neural networks are $\mathbb{Z}^n$ lattice for some $n \in \mathbb{N}$; nevertheless, it is known that the characteristic shape of neurons is tree. Some medical experiments suggest that the signal transmission might be blocked by those diseased neurons, which motivate the investigation of neural networks on ``incomplete" trees such as Fibonacci-Cayley tree.

% -------------------------------------------------------------
\subsection{Neural networks on Cayley tree}

To clarify the discussion, we focus on the neural networks proposed by Chua and Yang \cite{CY-ITCS1988}, which is known as cellular neural network and is widely applied to many disciplines such as signal propagation between neurons, pattern recognition, and self-organization; the elaboration of other models such as Hopfield neural networks can be carried out analogously.

Let $\Sigma^*$ be the set of nodes of Cayleys as described in the previous section. A \emph{cellular neural network on Cayley tree} (CTCNN) is represented as
\begin{equation}\label{eq:cnn-tree}
\dfrac{d}{dt} x_w(t) = - x_w(t) + z + \sum_{v \in \mathcal{N}} a_v f(x_{wv}(t)),
\end{equation}
for some finite set $\mathcal{N} \subset \Sigma^*$ known as the neighborhood, $v \in \mathcal{N}$, and $t\geq 0$. The transformation
\begin{equation}\label{eq:piecewise-linear}
f(s)=\dfrac{1}{2}(|s+1|-|s-1|)  
\end{equation}
is called the \emph{output function} or \emph{activation function}, and $z$ is called the \emph{threshold}. The weighted parameters $\mathbf{A} = (a_v)_{v \in \mathcal{N}}$ is called the \emph{feedback template}, and  Figure \ref{fig:TreeCNN} shows the connection of a binary CTCNN with the nearest neighborhood.

\begin{figure}[tbp]
\begin{center}
\includegraphics[scale=0.8,page=3]{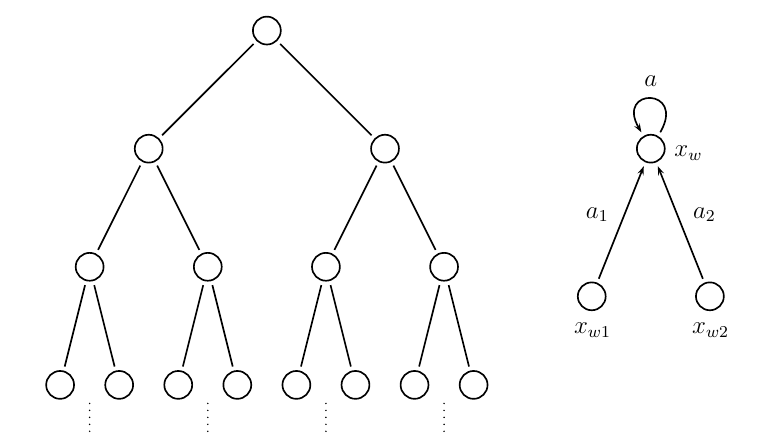}
\end{center}
\caption{A cellular neural network with the nearest neighborhood defined on binary trees. In this case, the neighborhood $\mathcal{N} = \{\epsilon, 0, 1\}$ and $a = a_{\epsilon}$.}
\label{fig:TreeCNN}
\end{figure}

A \emph{mosaic solution} $x = (x_w)_{w \in \Sigma^*}$ of \eqref{eq:cnn-tree} is an equilibrium solution which satisfies $|x_w| > 1$ for all $w \in \Sigma^*$; its corresponding pattern $y = (y_w)_{w \in \Sigma^*} = (f(x_w))_{w \in \Sigma^*}$ is called a \emph{mosaic output pattern}. Since the output function \eqref{eq:piecewise-linear} is piecewise linear with $f(s)=1$ (resp.~$-1$) if $s \geq 1$ (resp.~$s \leq -1$), the output of a mosaic solution $x = (x_w)_{w \in \Sigma^*}$ must be an element in $\left\{ -1,+1\right\}^{\Sigma^*}$, which is why we call it a \emph{pattern}. Given a CTCNN, we refer to $\mathbf{Y}$ as the output solution space; namely,
\begin{equation*}
\mathbf{Y} = \left\{ (y_w)_{w \in \Sigma^*}: y_w = f(x_w) \text{ and } (x_w)_{w \in \Sigma^*} \text{ is a mosaic solution of } \eqref{eq:cnn-tree} \right\} .
\end{equation*}

% -------------------------------------------------------------
\subsection{Learning problem}

Learning problem (also called the inverse problem) is one of the most investigated topics in a variety of disciplines. From the mathematical point of view, determining whether a given collection of output patterns can be exhibited by a CTCNN is essential for the study of learning problem. This subsection reveals the necessary and sufficient condition for the capability of exhibiting the output patterns of CTCNNs. For the compactness and consistency of this paper, we focus on cellular neural networks on the binary tree with the nearest neighborhood. The discussion of general cases is similar to the investigation in \cite{BC-NN2015,BCLL-JDE2009,Chang-ITNNLS2015}, thus it is omitted.

A CTCNN with the nearest neighborhood is realized as
\begin{equation}\label{eq:cnn-tree-nearest-nbd}
\dfrac{d}{dt} x_w(t) = - x_w(t) + z + a f(x_{w}(t)) + a_1 f(x_{w1}(t)) + a_2 f(x_{w2}(t)),
\end{equation}
where $a, a_1, a_2 \in \mathbb{R}$ and $w \in \Sigma^*$. Considering the mosaic solution $x = (x_w)_{w \in \Sigma^*}$, the necessary and sufficient condition for $y_w = f(x_w) = 1$ is
\begin{equation}\label{eq:cnn-state+}
a - 1 + z > -(a_1 y_{w1} + a_2 y_{w2}).
\end{equation}
Similarly, the necessary and sufficient condition for $y_w = f(x_w) = -1$ is
\begin{equation}\label{eq:cnn-state-}
a - 1 - z > a_1 y_{w1} + a_2 y_{w2}.
\end{equation}
Let
$$
V^2 = \{ \mathbf{v} \in \mathbb{R}^2 : \mathbf{v} = (v_1, v_2), \text{ and } |v_i| = 1, 1 \leq i \leq 2 \},
$$
and let $\alpha = (a_1, a_2)$ represent the feedback template without the self-feedback parameter $a$. The admissible local patterns with the $+1$ state in the parent neuron is denoted by
\begin{equation}\label{eq:cnn-state+no-center}
\widetilde{\mathcal{B}}_+( \mathbf{A}, z) = \{\mathbf{v} \in V^2: a - 1 + z > -\alpha \cdot v \},
\end{equation}
where ``$\cdot$" is the inner product in Euclidean space. Similarly, the admissible local patterns with the $-1$ state in the parent neuron is denoted by
\begin{equation}\label{eq:cnn-state-no-center}
\widetilde{\mathcal{B}}_-( \mathbf{A}, z) = \{\mathbf{v} \in V^2: a - 1 - z > \alpha \cdot v \}.
\end{equation}
Furthermore, the admissible local patterns induced by $(\mathbf{A}, z)$ can be denoted by
\begin{equation}
\mathcal{B}(\mathbf{A}, z) = \mathcal{B}_+( \mathbf{A}, z) \bigcup \mathcal{B}_-( \mathbf{A}, z),
\end{equation}
where
\begin{align*}
\mathcal{B}_+( \mathbf{A}, z) &= \{\mathbf{v}: v_{\epsilon} = 1 \text{ and } (v_1, v_2) \in \widetilde{\mathcal{B}}_+( \mathbf{A}, z)\}, \\
\mathcal{B}_-( \mathbf{A}, z) &= \{\mathbf{v}: v_{\epsilon} = -1 \text{ and } (v_1, v_2) \in \widetilde{\mathcal{B}}_-( \mathbf{A}, z)\}.
\end{align*}
Namely, $\mathcal{B}(\mathbf{A}, z)$ consists of two-blocks over $\mathcal{A} = \{1, -1\}$. For simplicity, we omit the parameters $(\mathbf{A}, z)$ and refer to $\mathcal{B}(\mathbf{A}, z)$ as $\mathcal{B}$. Meanwhile, we denote the output space $\mathbf{Y}$ by $\mathsf{X}_{\mathcal{B}}$ to specify the set of local patterns $\mathcal{B}$.

Suppose $U$ is a subset of $V^n$, where $n \geq 2 \in \mathbb{N}$. Let $U^c = V^n \setminus U$. We say that $U$ satisfies the \emph{linear separation property} if there exists a hyperplane $H$ which separates $U$ and $U^c$. More precisely, $U$ satisfies the separation property if and only if there exists a linear functional $g(z_1, z_2, \ldots, z_n) = c_1 z_1 + c_2 z_2 + \cdots + c_n z_n$ such that
$$
g(\mathbf{v}) > 0 \quad \text{for} \quad \mathbf{v} \in U \quad \text{and} \quad g(\mathbf{v}) < 0 \quad \text{for} \quad \mathbf{v} \in U^c.
$$
Figure \ref{fig:separation} interprets those $U \subset V^2$ satisfying the linear separation property.

\begin{figure}[tbp]
\begin{center}
\includegraphics[scale=0.8,page=4]{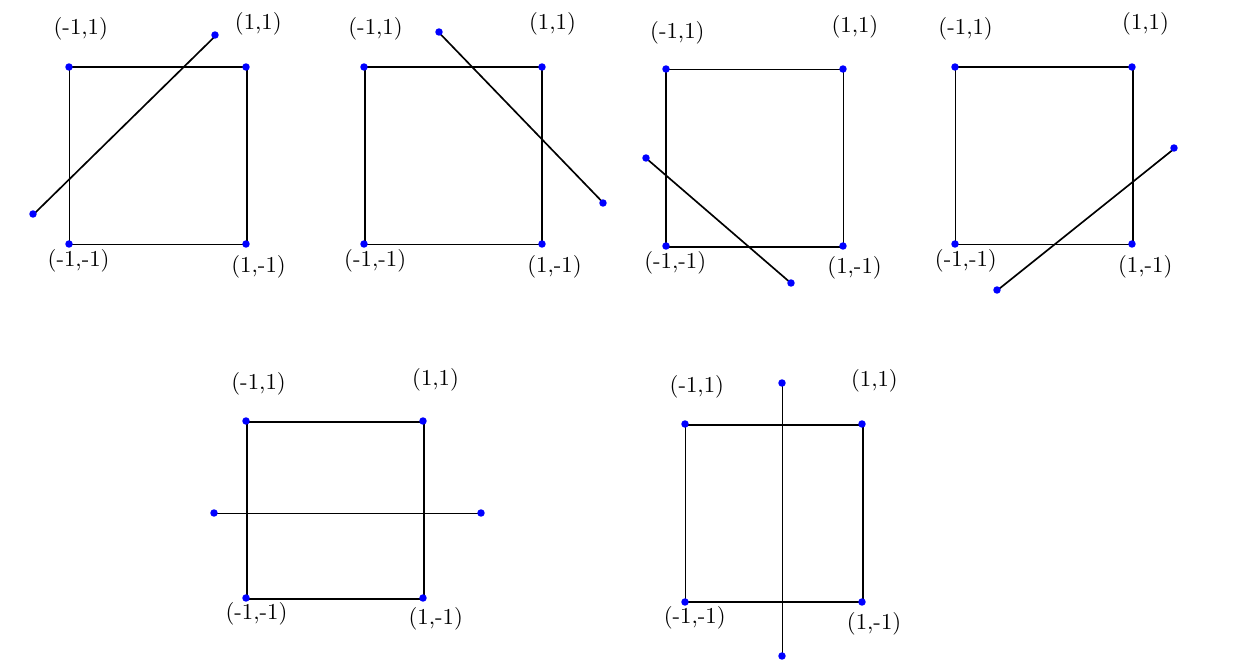}
\end{center}
\caption{Suppose $U$ is a proper subset of $V^2 = \{-1, 1\}^2$. There are only $12$ choices of $U$ that satisfies the linear separation property}
\label{fig:separation}
\end{figure}

Proposition \ref{prop:separation-property} elucidates the necessary and sufficient condition for the learning problem of CTCNNs, which follows from straightforward examination. Thus the proof is omitted. Such a property holds for arbitrary $\mathcal{N}$ provided $\mathcal{N}$ is prefix-closed. Readers are referred to \cite{BC-NN2015} for more details.

\begin{proposition}\label{prop:separation-property}
A collection of patterns $\mathcal{B} = \mathcal{B}_+ \bigcup \mathcal{B}_-$ can be realized in \eqref{eq:cnn-tree-nearest-nbd} if and only if either of the following conditions is satisfied:
\begin{enumerate}[\bf ({Inv}1)]
\item $-\widetilde{\mathcal{B}}_+ \subseteq \widetilde{\mathcal{B}}_-$ and $\widetilde{\mathcal{B}}_-$ satisfies linear separation property;
\item $-\widetilde{\mathcal{B}}_- \subseteq \widetilde{\mathcal{B}}_+$ and $\widetilde{\mathcal{B}}_+$ satisfies linear separation property.
\end{enumerate}
Herein, $\widetilde{\mathcal{B}}_+$ and $\widetilde{\mathcal{B}}_-$ are defined in \eqref{eq:cnn-state+no-center} and \eqref{eq:cnn-state-no-center}, respectively.
\end{proposition}

We emphasize that Proposition \ref{prop:separation-property} demonstrate the parameter space can be partitioned into finitely many equivalent regions. Indeed, whenever the parameters $a_1$ and $a_2$ are determined, \eqref{eq:cnn-state+} and \eqref{eq:cnn-state-} partition the $a$-$z$ plane into $25$ regions; the ``order" (i.e., the relative position) of lines $a - 1 +(-1)^{\ell} z = (-1)^{\ell} (a_1 y_{w1} + a_2 y_{w2})$, $\ell = 1, 2$, can be uniquely determined by the following procedures:
\begin{enumerate}[1)]
\item The signs of $a_1, a_2$ (i.e., the parameters are positive or negative).
\item The magnitude of $a_1, a_2$ (i.e., $|a_1| > |a_2|$ or $|a_1| < |a_2|$).
\end{enumerate}
This partitions $a$-$z$ plane into $8 \times 25 = 200$ regions. As a brief conclusion, the parameter space is partitioned into at most $200$ equivalent regions.

\begin{example}
Suppose that $0 < -a_1 < a_2$. It follows from $a_1 - a_2 < -a_1 - a_2 < a_1 + a_2 < -a_1 + a_2$ that, whenever $a$ and $z$ are fixed, the ``ordered'' basic set of admissible local patterns $\mathcal{B} = \mathcal{B}_+ \bigcup \mathcal{B}_-$ must obey
$$
\mathcal{B}_+ \subseteq \{(+; -, +), (+; +, +), (+; -, -), (+; +, -)\}
$$
and
$$
\mathcal{B}_- \subseteq \{(-; +, -), (-; -, -), (-; +, +), (-; -, +)\},
$$
herein, we denote a two-block $u$ by $(u_{\epsilon}; u_1, u_2)$ and use symbols ``$+$" and ``$-$" to represent $+1$ and $-1$, respectively, for clarity.

\begin{figure}[tbp]
\begin{center}
\includegraphics[scale=0.8,page=5]{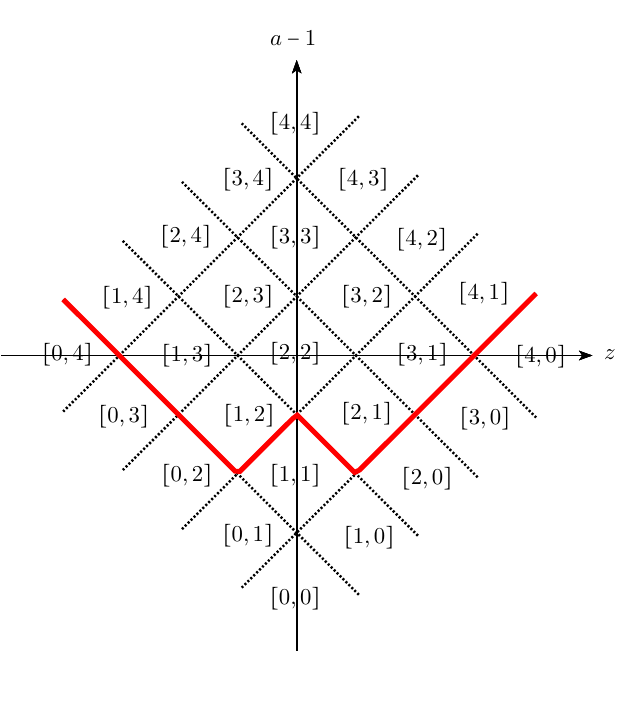}
\end{center}
\caption{For each chosen pair of parameters $(a_1, a_2)$, the $a$-$z$ plane is partitioned into $25$ equivalent regions.}
\label{fig:entropy-diagram-binary-tree}
\end{figure}

If the parameters $a$ and $z$ locate in the region $[3, 2]$(cf.~Figure \ref{fig:entropy-diagram-binary-tree}), then the set of feasible local patterns is
$$
\mathcal{B}_{[3, 2]} = \{(+; -, +), (+; +, +), (+; -, -), (-; +, -), (-; -, -)\}.
$$
A careful but straightforward examination indicates that two symbols are both inessential if and only if the pair of parameters $(a, z)$ is in the plane below the red $W$-shape line.
\end{example}

% -------------------------------------------------------------
\subsection{Neural networks on Fibonacci-Cayley tree}

The definition of cellular neural networks on Fibonacci-Cayley tree is similar to the definition of Fibonacci-Cayley tree-shifts. Let $A$ and $\Sigma_A$ be the same as defined in the previous section. A \emph{cellular neural network on Fibonacci-Cayley tree} (FTCNN) with the nearest neighborhood is realized as
\begin{equation}\label{eq:cnn-fibonacci-tree-nearest-nbd}
\dfrac{d}{dt} x_w(t) = \left\{\begin{aligned}
&- x_w(t) + z + a f(x_{w}(t)) + a_1 f(x_{w1}(t)) + a_2 f(x_{w2}(t)), && w_n = 1; \\
&- x_w(t) + z + a f(x_{w}(t)) + a_1 f(x_{w1}(t)), && w_n = 2;
\end{aligned}\right.
\end{equation}
where $a, a_1, a_2 \in \mathbb{R}$ and $w = w_1 \cdots w_n \in \Sigma_A$. Suppose the parameters $\mathbf{A} = (a, a_1, a_2)$ and $z$ are given. Let $\mathcal{B}$ be the set of feasible local patterns corresponding to $(\mathbf{A}, z)$ which is studied in the previous subsection. It can be verified without difficulty that the output space of \eqref{eq:cnn-fibonacci-tree-nearest-nbd} is
$$
\mathbf{Y}_F = \{t \in \mathcal{A}^{\Sigma_A}: t = t'|_{\Sigma_A} \text{ for some } t' \in \mathbf{Y} = \mathsf{X}_{\mathcal{B}}\}.
$$
In other words, the output space $\mathbf{Y}_F$ of \eqref{eq:cnn-fibonacci-tree-nearest-nbd} is a Markov-Fibonacci tree-shift.

% -------------------------------------------------------------
\subsection{Entropy of neural networks on Fibonacci tree}

It is of interest that how much information a diseased neural network can store. Based on the algorithm developed in the previous section, we study the entropy of FTCNNs with the nearest neighborhood.

Suppose that the parameters $a_1, a_2$ are given. A pair of parameters $(a, z)$ is called \emph{critical} if, for each $r > 0$, there exists $(a', z'), (a'', z'')$ such that $h(\mathsf{X}_{\mathcal{B}'}) \neq h(\mathsf{X}_{\mathcal{B}''}) = 0$, where $\mathcal{B}' = \mathcal{B}(a', a_1, a_2, z'), \mathcal{B}'' = \mathcal{B}(a'', a_1, a_2, z'')$. Notably, the $a$-$z$ plane is partitioned into $25$ equivalent regions whenever $a_1$ and $a_2$ are given and is indexed as $[p, q]$ for $0 \leq p, q \leq 4$; we use $\mathbf{Y}_{[p, q]}$ instead of $\mathsf{X}_{\mathcal{B}}$ for simplicity.

\begin{theorem}\label{thm:entropy-region-W-equation}
Suppose that the parameters $a_1, a_2$ are given. Let $m = \min\{|a_1|, |a_2|\}, M = \max\{|a_1|, |a_2|\}$. Then
$$
h(\mathbf{Y}_{[p, q]}) = \left\{\begin{aligned}
& 0, && \hbox{if $\min\{p, q\} = 0$ or $\max\{p, q\} = 1$;} \\
& \ln g, && \hbox{otherwise.}
\end{aligned}\right.
$$
Furthermore, $(a, z)$ is critical if and only if
\begin{equation}\label{eq:W-shape}
a - 1 = ||z| - m| - M.
\end{equation}
\end{theorem}
\begin{proof}
Obviously, $h(\mathbf{Y}_{[p, q]}) = 0$ if $\min\{p, q\} = 0$ or $\max\{p, q\} = 1$. We only need to show that $h(\mathbf{Y}_{[p, q]}) = \ln g$ provided $p, q > 0$ and $\max\{p, q\} \geq 2$. Actually, it suffices to consider the cases where $[p, q] = [1, 2]$ or $[p, q] = [2, 1]$ since the entropy function is increasing. We demonstrate that $h(\mathbf{Y}_{[1, 2]}) = \ln g$; the other case can be derived analogously, thus it is omitted.

Suppose that both symbols in $\mathcal{A}$ are essential. Let
$$
\theta_n = (\ln \gamma_{1;n}, \ln \gamma_{2;n}, \ln \gamma_{1;n-1}, \ln \gamma_{2;n-1})^T,
$$
where $\gamma_{1;n}$ (resp.~$\gamma_{2;n}$) denotes the number of $n$-blocks whose rooted symbol is $+$ (resp.~$-$). Notably, for each simple recurrence representation of $\mathbf{Y}_{[1, 2]}$, there exist $2 \times 2$ binary matrices $B_1$ and $B_2$ satisfying $\sum\limits_{k=1}^2 B_i (j, k) = 1$ for $1 \leq i, j \leq 2$ and
$M =  \begin{pmatrix}
B_1 & B_2 \\
I_2 & 0_2
\end{pmatrix}$
such that $\theta_n = M \theta_{n-1}$ for $n \geq 2$. Set $\mathbf{v} = (g, g, 1, 1)^T$. Then
$$
M \mathbf{v} = (g+1, g+1, g, g)^T = (g^2, g^2, g, g)^T = g \mathbf{v}.
$$
Perron-Frobenius Theorem and Theorem \ref{thm:entropy-Fibonacci-tree-shift} imply that $h(\mathbf{Y}_{[1, 2]}) \geq \ln g$. It is easily seen that $h(\mathbf{Y}_{[1, 2]}) \leq \ln g$, thus we can conclude that $h(\mathbf{Y}_{[1, 2]}) = \ln g$.

Suppose that there is exactly one essential symbol. (In this case, the essential symbol is $-$.) It can be verified from \eqref{eq:cnn-state+}, \eqref{eq:cnn-state-}, and Proposition \ref{prop:separation-property} that
$$
\mathcal{B}_{[1, 2]} = \{(+; +, +), (-; -, -), (-; -, +)\}
$$
or
$$
\mathcal{B}_{[1, 2]} = \{(+; +, +), (-; -, -), (-; +, -)\}.
$$
For either case,
$$
\gamma_{i;n} = \gamma_{i;n-1} \gamma_{i;n-2} \quad \text{for} \quad i = 1, 2,
$$
is a simple recurrence representation of $\mathbf{Y}_{[1, 2]}$, and its corresponding adjacency matrix is
$$
M = \begin{pmatrix}
1 & 0 & 1 & 0 \\
0 & 1 & 0 & 1 \\
1 & 0 & 0 & 0 \\
0 & 1 & 0 & 0
\end{pmatrix}.
$$
Deleting the first and third rows and columns of $M$ delivers
$M' = \begin{pmatrix}
1 & 1 \\
1 & 0
\end{pmatrix}$. This shows that
$$
h(\mathbf{Y}_{[1, 2]}) \geq \ln \rho_{M'} = \ln g,
$$
which is followed by $h(\mathbf{Y}_{[1, 2]}) = \ln g$.

Next, we show that $(a, z)$ is critical if and only if $(a, z)$ satisfies \eqref{eq:W-shape}. Let $C = \{\ell_1 a_1 + \ell_2 a_2: \ell_1, \ell_2 \in \{-1, 1\}\}$, and let
$$
K_1 = \max C \quad \text{and} \quad K_2 = \max C \setminus \{K_1\}.
$$
be the largest and the second largest elements in $C$, respectively. A careful but straightforward verification asserts that $(a, z)$ is critical if and only if
$$
a - 1 = \left| |z| -\dfrac{K_1 - K_2}{2}\right| - \dfrac{K_1 + K_2}{2}.
$$
(See Figure \ref{fig:entropy-diagram-binary-tree} for more information.) The desired result follows from the fact that
$$
K_1 = |a_1| + |a_2| \quad \text{and} \quad K_2 = |a_1| + |a_2| - 2m.
$$

This completes the proof.
\end{proof}

% -------------------------------------------------------------
\section{Conclusion}

This paper studies the entropy of Fibonacci-Cayley tree-shifts, which are shift spaces defined on Fibonacci-Cayley trees. Entropy is one of the most frequently used invariant that reveals the growth rate of patterns stored in a system. Followed by demonstrating that the computation of the entropy of a Fibonacci tree-shift of finite type is equivalent to studying a nonlinear recursive system, we propose an algorithm for computing entropy. It is seen that the so-called simple recursive system plays an important role. 

As an application, we elucidate the complexity of neural networks whose topology is Fibonacci-Cayley tree. Such a model reflects a brain going with neuronal dysfunction such as Alzheimer's disease. A Fibonacci tree-shift came from a neural network is constrained by the so-called separation property. Aside from demonstrating a surprising phenomenon that there are only two possibilities of entropy for neural networks on Fibonacci-Cayley tree, we reveal the formula of the boundary in the parameter space.

A related interesting problem is ``does the black-or-white entropy spectrum of neural networks still hold for other network topologies?" Further discussion is in preparation.

% bibliography ---------------------------------------------------
\bibliographystyle{amsplain}%amsplain, amsalpha, abbrvnat, alpha, ieeetr, abbrv, unsrt, acm, plainnat, plain, siam
\bibliography{../../grece}

% -------------------------------------------------------------
\end{document}